\documentclass[a4paper,10pt]{article}
\usepackage{hyperref} 
\usepackage[english]{babel}
\usepackage[T1]{fontenc}
\usepackage[dvips]{graphicx}
\usepackage{amsmath, amsfonts,amsthm, amssymb}
\usepackage{verbatim}

\usepackage{enumerate}
\usepackage[latin1]{inputenc}
\usepackage{geometry}
\usepackage{color}
\usepackage{mathrsfs}

\def\Im{{\rm Im}}

\def\beq{\begin{equation}}   \def\eeq{\end{equation}}
\def\bea{\begin{eqnarray}}  \def\eea{\end{eqnarray}}

\newcommand{\be}{\begin{equation}}

\newcommand{\ee}{\end{equation}}

\newcommand{\ben}{\begin{equation*}}
\newcommand{\een}{\end{equation*}}
\newcommand{\ban}{\begin{align*}}
\newcommand{\ean}{\end{align*}}

\def\vf{\varphi}

\def\bee{{\bf e}}
\def\be#1{\bee_{#1}}

\def\ir{{\rm i}}

\def\vf{\varphi}

\newcommand{\1}{\mathbb  I}

\def\norma#1{\left\|#1\right\|}

\newcommand{\C}{{\mathbb C}}
\newcommand{\D}{{\mathcal D}}
\newcommand{\E}{{\mathcal E}}

\newcommand{\N}{{\mathbb N}}

\newcommand{\R}{{\mathbb R}}

\newcommand{\T}{{\mathbb T}}
\newcommand{\Z}{{\mathbb Z}}

\newcommand{\cE}{{\mathcal E}}

\newcommand{\cH}{{\mathcal H}}

\newcommand{\cL}{{\mathcal L}}

\newcommand{\cO}{{\mathcal O}}
\newcommand{\cP}{{\mathcal P}}

\newcommand{\cR}{{\mathcal R}}

\newcommand{\cU}{{\mathcal U}}

\newcommand{\norm}[1]{\| #1 \|}

\newcommand{\im}{{\rm i}}

\def\uno{{\bf 1}}

\numberwithin{equation}{section}
\newtheorem{theorem}{Theorem}[section]

\newtheorem{lemma}[theorem]{Lemma}
\newtheorem{corollary}[theorem]{Corollary}
\newtheorem{proposition}[theorem]{Proposition}
\newtheorem{definition}[theorem]{Definition}
\newtheorem{remark}[theorem]{Remark}

\def\ex#1{e^{#1 JA_j}}
\def\ps{{\mathfrak{p}}}

\newcommand{\eps}{\epsilon}

\newcommand{\dd}{  \text{d}   }

\newcommand{\om}{  \omega   }

\renewcommand{\a}{  \alpha   }
\newcommand{\p}{  \partial   }
\renewcommand{\b}{  \beta   }

\newcommand{\s}{  \sigma   }
\newcommand{\lan}{  \langle  }
\newcommand{\ran}{  \rangle  }
\newcommand{\ka}{  \kappa   }

\renewcommand{\L}{  \mathcal{L}   }

\newcommand{\M}{  \mathcal{M}   }

\newcommand{\diag}{\operatorname{diag}}
\newcommand{\meas}{\operatorname{meas}}

\newcommand{\dist}{\operatorname{dist}}

\def\norma#1{\left\| #1\right\|}

\title{Reducibility of the Quantum Harmonic Oscillator in
  $d$-dimensions with Polynomial Time Dependent Perturbation
}

\author{
D. Bambusi
\footnote{Dipartimento di Matematica, Universit\`a degli Studi di Milano, Via Saldini 50, I-20133
Milano. 
\newline
{\em Email: {\tt dario.bambusi@unimi.it}}},
B. Gr\'ebert
\footnote{Laboratoire de Math\'ematiques Jean Leray, Universit\'e de Nantes, 2 rue de la Houssini\`ere
BP 92208, 44322 Nantes.
\newline
 {\em Email: {\tt benoit.grebert@univ-nantes.fr}}}, 
 A. Maspero
 \footnote{
  International School for Advanced Studies (SISSA), Via Bonomea 265, 34136, Trieste, Italy \newline
 \textit{Email: } \texttt{alberto.maspero@sissa.it}}, 
 D. Robert
 \footnote{
 Laboratoire de Math\'ematiques Jean Leray, Universit\'e de Nantes, 2 rue de la Houssini\`ere
BP 92208, 44322 Nantes.
\newline
 {\em Email: {\tt  didier.robert@univ-nantes.fr}}}
}

\begin{document}

\maketitle

\begin{abstract}
We prove a reducibility result for a quantum harmonic oscillator in arbitrary
dimension with arbitrary frequencies perturbed by a linear operator
which is a polynomial of degree two in $x_j$, $-\ir\partial_j$ with
coefficients which depend quasiperiodically on time.
\end{abstract}

\section{Introduction and statement}

The aim of this paper is to present a reducibility result for the time
dependent Schr\"odinger equation
\begin{align}
\label{schro}
\ir \dot\psi=H_\epsilon(\omega t)\psi\ , \ x\in\R^d
\\
\label{H}
H_\epsilon(\omega t):=H_0+\epsilon W(\omega t, x,-\ir \nabla)
\end{align}
where 
\begin{equation}
\label{pot}
H_0:= -\Delta+V(x),\; \quad V(x):=\sum_{j=1}^{d}\nu_j^2x_j^2\ ,\quad \nu_j>0
\end{equation}
and $W( \theta,  x,\xi )$ {\it is a real  polynomial in $(x,\xi)$ of degree  at most two}, with
coefficients being real analytic functions of $\theta\in \T^n$. Here
$\omega$ are parameters which are assumed to belong to the set
$\D=(0,2\pi)^{n}$.

For $\epsilon=0$ the spectrum of \eqref{H} is
given by
\begin{equation}
\label{sp}
\sigma(H_0)=\{\lambda_k\}_{k \in \N^d}\ ,\quad
\lambda_k\equiv\lambda_{(k_1,...,k_d)}:=\sum_{j=1}^{d}(2k_j+1)\nu_j\ ,
\end{equation}
with $k_j\geq 0$ integers. In particular if the frequencies
$\nu_j$ are nonresonant, then the differences between couples of
eigenvalues are dense on the real axis. As a consequence, in the case  $\epsilon =0$ most of the
solutions of \eqref{schro} are almost periodic with an infinite number of rationally independent frequencies.

Here we will prove  that for any choice of the mechanical frequencies
$\nu_j$ and for $\omega$ belonging to a set of large measure in $\D$ the system \eqref{schro}
is reducible: precisely there exists a time quasiperiodic unitary
transformation of $L^2(\R^d)$  which conjugates \eqref{H} to a time independent
operator; we also deduce boundedness of the Sobolev norms of the solution.

The proof exploits the fact that for polynomial Hamiltonians of degree
at most 2 the correspondence between classical and quantum mechanics
is exact (i.e. without error term), so that the result can be proven by
exact quantization of the classical KAM theory which ensures
reducibility of the classical Hamiltonian system
\begin{equation}
\label{clas}
h_\epsilon := h_0+\epsilon W( \omega t, x,\xi)\ ,\quad
h_0:=\sum_{j=1}^{d}{\xi_j^2+\nu_j^2x_j^2}\ .
\end{equation} 
We will use (in the appendix) the exact correspondence between
classical and quantum dynamics of quadratic Hamiltonians also to prove
a complementary result. Precisely we will present a class of examples
(following \cite{graffi}) in which one generically has growth of Sobolev
norms. This happens when the frequencies $\omega$ of the external
forcing are resonant with some of the $\nu_j$'s.

We recall that the exact correspondence between classical and quantum
dynamics of quadratic Hamiltonians was already exploited in the paper
\cite{Hagedorn} to prove stability/instability results for one degree
of freedom time dependent quadratic Hamiltonians.

Notwithstanding the simplicity of the proof, we think that the present
result could have some interest, since this is the first example of a
reducibility result for a system in which the gaps of the unperturbed
spectrum are dense in $\R$. Furthermore it is one of the few cases in
which reducibility is obtained for systems in more than one space
dimension.

Indeed, most of the results on the reducibility problem for 
\eqref{schro} have been obtained in the one dimensional case, and also
the results in higher dimensions obtained up to now deal only with
cases in which the spectrum of the unperturbed system has gaps whose
size is bounded from below, like in the Harmonic oscillator (or in
the Schr\"odinger equation on $\T^d$). On the other hand we restrict
here to perturbations, which although unbounded, must belong to the
very special class of polynomials in $x_j$ and $-\ir\partial_j$. 
The reason is that for operators in this class, the commutator  is the operator whose symbol is the Poisson bracket of the corresponding symbols,  without any error term (see Remark \ref{key.1}  and Remark \ref{key.31}).  In order to deal with more general perturbations one needs
further ideas and techniques.

\vskip10pt

Before closing this introduction we recall some previous works on the
reducibility problem for \eqref{schro} and more generally for
perturbations of the Schr\"odinger equation with a potential
$V(x)$. As we already anticipated, most of the works deal with the one
dimensional case. The first one is \cite{C87} in which pure point
nature of the Floquet operator is 
obtained in case of a smoothing perturbation of the Harmonic
oscillator in dimension 1 (see also \cite{Kuk93}). The techniques of
this paper were extended in \cite{DS96,DSV02}, in order to deal with
potentials growing superquadratically (still in dimension 1) but with
perturbations which were only required to be bounded.

A slightly different approach originates from the so called KAM
theory for PDEs \cite{Kuk87,Way90}. In particular the methods
developed in that context in order to deal with unbounded perturbations
(see \cite{Kuk97,Kuk98}) where exploited in \cite{BG01} in order to
deal with the reducibility problem of \eqref{schro} with
superquadratic potential in dimension 1 (see \cite{LY10} for a further
improvement). The case of bounded perturbations of the Harmonic
oscillator in dimension 1 was treated in \cite{W08,GT11}.

The only works dealing with the higher dimensional case are \cite{eliasson09}
actually dealing with bounded perturbations of the Schr\"odinger
equation on $\T^d$ and \cite{BP16} dealing with bounded perturbations
of the completely resonant Harmonic oscillator in $\R^d$. 

All these papers  deal with cases where the spectrum of the
unperturbed operator is formed by well separated eigenvalues.  In the
higher dimensional cases they are allowed to have high multiplicity {blue} localized in clusters.
But then the perturbation must have special properties ensuring that
the clusters are essentially not destroyed under the KAM iteration.

Finally we recall the works \cite{Bam16I,Bam17b} in which
pseudodifferential calculus was used together with KAM theory in
order to prove reducibility results for \eqref{schro} (in dimension 1)
with unbounded perturbations. The ideas of the present paper are a
direct development of the ideas of  \cite{Bam16I,Bam17b}. We also
recall that the idea of using pseudodifferential calculus together
with KAM theory in order to deal with problems involving unbounded
perturbations originates from the work \cite{PT,IPT} and has been
developed in order to give a quite general theory in
\cite{BBM,BM,M} (see also \cite{FP}). \\

In order to state our main result, we need some preparations.
It is well  known that  the  equation (\ref{schro}) is well posed (see for example \cite{maro}) in  the scale $\cH^s$, $s\in\R$ of the weighted Sobolev spaces defined as follows.  For $s\geq 0$ let
$$\cH^s:=\{\psi\in L^2(\R^d)\colon \;\; H_0^{s/2}\psi\in L^2(\R^d)\} \ ,  $$
equipped with the natural Hilbert space norm $\norm{\psi}_s := \norm{H_0^{s/2} \psi}_{L^2(\R^d)}$.   
 For $s<0$, $\cH^s$ is defined 
  by duality.  Such  spaces are  not dependent on $\nu$
for  $\nu_j>0$, $1\leq j\leq d$. We also have $\cH^s \equiv{\rm Dom}(-\Delta +\vert x\vert^2)^{s/2}$.\\

We will prove the following reducibility theorem:
\begin{theorem}
\label{m.1}
Let $\psi$ be a solution of (\ref{schro}). There exist $\epsilon_*>0$, $C>0$ and
$\forall \left|\epsilon\right|<\epsilon_*$ a closed set
$\E_\epsilon\subset (0,2\pi)^n$  with $\meas((0,2\pi)^n \setminus \E_\epsilon) \leq C\epsilon^{\frac{1}{9}} $ and, $\forall
\omega \in \E_\epsilon $ there exists a unitary (in $L^2$) time
quasiperiodic map $U_{\omega}(\omega t)$ s.t. defining  $\varphi$  by
$U_{\omega}(\omega t)\varphi=\psi $, it satisfies the equation
\begin{equation}
\label{rido}
\ir \dot \varphi= H_{\infty}\varphi\ ,
\end{equation}  
with
  $H_{\infty}$
 a positive definite time independent operator which is unitary equivalent to a diagonal operator $$\sum_{j=1}^{d}\nu_j^{\infty}(x_j^2-\partial_{x_j}^2),$$
 where
$\nu_j^{\infty}=\nu_j^{\infty}(\omega)$ are defined for $\omega \in \E_\epsilon$ and fulfill the estimates
$$
|\nu_j-\nu_j^\infty|\leq C\epsilon \ , \qquad j = 1, \ldots, d  \ . 
$$ Finally the following properties hold
\begin{itemize}
\item[(i)] $\forall s \geq 0$,
$\forall \psi \in \cH^s$, $\theta \mapsto U_\om(\theta)\psi \in
C^0(\T^n; \cH^s)$.
\item[(ii)] $\forall s \geq 0$, $\exists C_s >0$ such that  for all $\theta\in\T^n$
\begin{equation} 
\label{pro1}
\norma{\uno-U_{\omega}(\theta)}_{\cL(\cH^{s+2};\cH^{s})}\leq C_s\epsilon.
\end{equation}
\item[(iii)]
$\forall s, r\geq0$, the map $\theta\mapsto
U_\omega(\theta)$ is of class $C^r(\T^n; \cL(\cH^{s+4r+2};\cH^{s}))$.
\end{itemize}
\end{theorem}

\begin{remark}
  \label{meta}
{Remark that in Theorem \ref{m.1},  if the  frequencies $\nu_j$ are resonant, then the change of coordinates $U_\omega$ is close to the identity  (in the sense of \eqref{pro1}), but  the Hamiltonian $H_\infty$ is not necessary diagonal.
However it is always possible to diagonalize it by means of a metaplectic transformation which is not close to the identity,  see Theorem \ref{KAMclassico} and Remark \ref{small} below.}
\end{remark}

Let  us denote by $\cU_{\epsilon, \omega}( t,\tau)$ the propagator generated by \eqref{schro}  such that $\cU_{\epsilon, \omega}(\tau, \tau)=\uno$, $\forall \tau \in \R$.
An immediate consequence of Theorem \ref{m.1} is that we have a Floquet   decomposition:
\beq\label{decom}
 \cU_{\epsilon,\omega}( t,\tau) =  U^*_\omega(\omega t){\rm e}^{-\im(t-\tau)H_\infty} U_\omega(\omega\tau).
\eeq

An other consequence of \eqref{decom} is  that   for any $s >0$ the norm $\norm{\cU_{\epsilon, \omega}(t,0) \psi_0}_s$  is bounded uniformly in time:
\begin{corollary}
\label{cor.1}
Let $\om \in \E_\eps$ with  $\vert\epsilon\vert <\epsilon_{*}$.  The following is true: for any $s >0$ one has \\
\beq\label{unest}
 c_{s}\Vert\psi_0\Vert_s \leq\Vert  \cU_{\epsilon,\omega}( t,0)\psi_0\Vert_s \leq C_{s}\Vert\psi_0\Vert_s \ ,\qquad \forall t\in\R \ , 
 \forall\psi_0\in\cH^s, 
\eeq
for some  $ c_{s}>0, C_{s}>0$.\\
Moreover there exists a constant $c'_{s}$ s.t. if the initial data $\psi_0 \in \cH^{s+2}$ then
\beq\label{unest2}
\Vert\psi_0\Vert_s  - \eps c'_{s} \norm{\psi_0}_{s+2}  \leq \Vert  \cU_{\epsilon,\omega}( t,0)\psi_0\Vert_s \leq \Vert\psi_0\Vert_s + \eps c'_{s} \norm{\psi_0}_{s+2} \ , \qquad \forall t\in\R \ .
\eeq
\end{corollary}

It is interesting to compare  estimate \eqref{unest} with the corresponding estimate which can be obtained for more general perturbations  $W(t,x,D)$. So denote  by $\cU(t,\tau)$ the  propagator of $H_0 + W(t,x,D)$ with $\cU(\tau, \tau) = \uno$. 
Then in \cite{maro} it is proved that if
 $W(t,x,\xi)$ is a real polynomial in $(x, \xi)$ of degree at most 2, the propagator $\cU(t,s)$ exists,  belongs to $\cL(\cH^s)$ $\, \forall s \geq 0$ and fulfills
$$
\norm{\cU(t,0)\psi_0}_s \leq e^{C_s |t|} \norm{\psi_0}_s \ , \qquad \forall t  \in \R
$$
(the estimate is sharp!).  If $W(t,x,\xi)$ is a polynomial of degree at most 1 one has
$$
\norm{\cU(t,0)\psi_0}_s \leq  C_s (1+ |t|)^{s} \, \norm{\psi_0}_s \ , \qquad \forall t \in \R \ . 
$$
Thus estimate \eqref{unest} improves dramatically the upper bounds proved in \cite{maro} when  the perturbation is small and depends quasiperiodically in time with "good" frequencies.

As a final remark we recall that  growth of Sobolev norms can indeed
happen if the frequencies $\om$ are not well chosen. In
 Appendix \ref{example}, we show that the Schr\"odinger equation
$$
\im \dot \psi =
\left[-\frac{1}{2}\partial_{xx}+\frac{x^2}{2}+ax\sin\omega t \right]
\psi \ , \qquad x \in \R \ 
$$ (which was already studied by Graffi and Yajima in \cite{graffi}
 who showed that the corresponding Floquet operator has continuous spectrum)
exhibits growth of Sobolev norms if and only if $\omega = \pm 1$,
which are clearly resonant frequencies. 
We also slightly generalize
the example.

An other  example  of  growth of Sobolev norms for the perturbed harmonic oscillator is given by Delort \cite{del1}. There the perturbation is a  pseudodifferential operator of order 0, periodic in time with resonant frequency $\omega =1$.

\begin{remark}
The uniform time estimate given in (\ref{unest}) is similar to the
main result obtained in \cite{eliasson09} for small perturbation of
the Laplace operator on the torus $\T^d$. Concerning perturbations of
harmonic oscillators in $\R^d$ most reducibility known results are
obtained for $d=1$ excepted in \cite{BP16}. 
\end{remark}

\begin{remark} In \cite{eliasson09,BP16} the estimate \eqref{unest2} is proved without
    loss of regularity; this is due to the fact that the perturbations
    treated in \cite{eliasson09,BP16} are bounded operators. There are
    also some cases (see e.g. \cite{BG01}) in which the reducing
    transformation is bounded notwithstanding the fact that the
    perturbation is unbounded, but this is due to the fact that the
    unperturbed system has suitable gap properties which are not
    fulfilled in our case.
  \end{remark}
  \begin{remark}
    \label{mesure}
The $\epsilon^{1/9}$ estimate on the measure of the set of resonant
frequencies is not optimal. We wrote it just for the sake of giving a
simple quantitative estimate. 
  \end{remark}

\begin{remark}\label{dynquant}
Denote by $\{\psi_{k}\}_{k \in \N^d}$ the set of Hermite functions,
namely the eigenvectors of $H_0$: $H_0 \psi_k = \lambda_k
\psi_k$. They form an orthonormal basis of $L^2(\R^d)$, and writing
$\psi = \sum_k c_k \psi_k$ one has $\norm{\psi}_{s}^2 \simeq \sum_k
(1+|k|)^{2s} |c_k|^2$. Denote $\displaystyle{\psi(t) = \sum_{k
    \in\N^d}c_k(t) \psi_k}$ the solution of \eqref{schro} written on
the Hermite basis.  Then \eqref{unest} implies the following {\em
  dynamical localization} for the energy of the solution: $\forall s
\geq 0$, $\exists \, C_s\equiv C_s(\psi_0) >0$:
\beq\label{ineq:dynquant} \sup_{t \in \R} |c_k(t)| \leq
C_s(1+|k|)^{-s} \, \ , \quad \forall k \in \N^d\ .  \eeq
\end{remark}
 From the dynamical property \eqref{ineq:dynquant} one  obtains easily that every state $\psi\in L^2(\R^d)$ is a bounded state for the time evolution 
$ \cU_{\epsilon,\omega}( t,0)\psi$  under the conditions of Theorem \ref{m.1} on $(\epsilon,\omega)$.  The corresponding definitions are given in \cite{enve}:

\begin{definition}[See \cite{enve}]
A function $\psi\in L^2(\R^d)$ is a bounded state (or belongs to the point spectral  subspace of  $\{ \cU_{\epsilon,\omega}( t,0)\}_{t\in\R}$)
 if the quantum trajectory
$\{\cU_{\epsilon,\omega}( t,0)\psi:\;\; t\in\R\}$ is a precompact
subset of $L^2(\R^d)$.
 \end{definition}

\begin{corollary}
\label{cor.din.loc}
Under the conditions of Theorem \ref{m.1} on $(\epsilon,\omega)$,
every state $\psi\in L^2 (\R^d)$ is a bounded state of
$\{ \cU_{\epsilon,\omega}( t,0)\}_{t\in\R}$.
\end{corollary}
\begin{proof}
To  prove that every state $\psi\in L^2(\R^d)$  is a bounded state for the time   evolution
 $\cU_{\epsilon,\omega}( t,0)\psi$, 
  using that $\cH^s$ is dense in $L^2(\R^d)$, it is enough to assume that $\psi\in\cH^s$, with $s>\frac{d}{2}$.
 With the notations of Remark \ref{dynquant}, we write 
 $$
 \psi(t) = \psi^{(N)}(t) + R^{(N)}(t),
 $$
 where 
 $\displaystyle{\psi^{(N)}(t) = \sum_{\vert k\vert\leq N}c_k(t)\psi_k}$ and  $\displaystyle{ R^{(N)}(t) = \sum_{\vert k\vert > N}c_k(t)\psi_k}$.\\
Take $\delta>0$. 
Applying \eqref{dynquant}, taking $N$ large enough,  we   get that  for all $t\in\R$,  $\Vert R^{(N)}(t)\Vert_0\leq \frac{\delta}{2}$. 
But $\{\psi^{(N)}(t),\;t\in\R\}$ is a subset of a finite dimensional linear space. So we get that $\{\cU_{\epsilon,\omega}( t,0)\psi:\;\; t\in\R\}$   is a precompact subset of $L^2(\R^d)$. \end{proof}

This last dynamical result is deeply connected with the spectrum of the  Floquet operator. First remark that  Theorem \ref{m.1} implies the following
\begin{corollary}
\label{r.flo.1}
The operator $U_{\omega}$ induces a unitary transformation
$L^2(\T^n)\otimes L^2(\R^d)$ which transforms the Floquet operator $K$, namely
$$
K:=-\ir \omega\cdot\frac{\partial}{\partial\theta}+ H_0+\epsilon W(\theta)
\ ,
$$
into 
$$
-\ir \omega\cdot\frac{\partial}{\partial\theta}+ H_\infty\ .
$$
Thus one has that the spectrum of $K$ is pure point and its eigenvalues are
$\lambda_j^{\infty}+\omega\cdot k$.
\end{corollary}

Notice that  Enss and Veselic proved   that  the  spectrum of the Floquet operator is pure point if and only if  every state is a bounded state \cite[Theorems 2.3 and Theorem 3.2]{enve}. So Corollary \ref{r.flo.1} gives another proof of Corollary \ref{cor.din.loc}.

\vspace{2em}
\noindent{\bf Acknowledgments.} The last three authors are supported
 by ANR -15-CE40-0001-02  ``BEKAM'' of the Agence Nationale de la Recherche.


\section{Proof}\label{proof}

To start with we scale the variables $x_j$ by defining
$x_j'=\sqrt{\nu_j} x_j$ so that, defining 
$$
h_j(x_j,\xi_j):=\xi_j^2+x_j^2\ ,\quad H_j:=-\partial^2_{x_j}+x_j^2\ ,
$$
one has 
\begin{equation}
\label{hj}
h_0=\sum_{j=1}^{d}\nu_j h_j\ ,\quad H_0=\sum_{j=1}^{d}\nu_jH_j\ .
\end{equation}
\begin{remark}
Notice that  for any  positive definite quadratic Hamiltonian $h$ on $\R^{2d}$ there  exists  a symplectic basis such that 
$h=\sum_{j=1}^{d}\nu_j h_j$, with $\nu_j >0$   for  $1\leq j\leq d$ (see \cite{horm3}).
\end{remark}

For convenience in this paper  we shall consider the
Weyl quantization. The Weyl quantization of a  symbol $f$ is the operator ${\rm Op}^w(f)$, defined as usual as
$$
{\rm Op}^w(f) u(x) = \frac{1}{(2\pi)^d} \int_{y,\xi \in \R^d} \, e^{\im (x-y)\xi} \, f\left(\frac{x+y}{2},\xi\right) \, u(y) \, dy \, d\xi \ .
$$
Correspondingly we will say that an operator $F= {\rm Op}^w(f) $   is the  Weyl
operator  with Weyl symbol $f$.
Notice that for polynomials $f$   of degree at most 2 in $(x,\xi)$,  $ {\rm Op}^w(f)=f(x, D) +  const$, 
where $D= \im^{-1}\nabla_x$.

Most of the times we also use the notation $f^w(x,D):= {\rm Op}^{
  w}(f)$.  In particular, in equation (\ref{H}) $W(\omega t, x,-\ir
\partial_x)$ denotes the Weyl operator $W^w(\omega t,x,D)$.\\ Given a
Hamiltonian $\chi=\chi(x,\xi)$, we will denote by
$\phi^t_{\chi}$ the flow of the corresponding classical Hamilton
equations.

It is well known that, if $f$ and $g$ are symbols, then the operator
$-\ir[f^w(x,D);g^w(x,D)]$ admits a symbol denoted by $\left\{f;g\right\}_M$ (Moyal bracket). 
Two fundamental properties of quadratic polynomial symbols are given
by the following well known remarks.

\begin{remark}
\label{key.1}
{If $f$ or $g$ is a polynomial of
degree at most 2, then $\left\{f;g\right\}_M=\left\{f;g\right\}$,
where 
$$\left\{f;g\right\}:=\sum_{j=1}^{d}\frac{\partial f}{\partial x_j}\frac{\partial g}{\partial \xi_j}-\frac{\partial g}{\partial x_j}\frac{\partial f}{\partial \xi_j}
$$ is the Poisson Bracket of $f$
and $g$. }
\end{remark}
\begin{remark}
\label{key.2}
Let $\chi$ be a polynomial of degree at most 2, then it follows from the
previous remark that, for any Weyl operator $f^w(x,D)$, the symbol of
$e^{\ir t \chi^w(x,D)}f^w(x,D)e^{-\ir t \chi^w(x,D)}$ is $f\circ \phi_{\chi}^t$.
\end{remark}

\begin{remark}\label{key.31} If $f$ and $g$ are not quadratic
    polynomials, then $\{ f; g\}_M = \{f; g\} +${\rm lower order terms};
    similar lower order corrections would appear in the symbol of
    $e^{-\ir t \chi^w(x,D)}f^w(x,D)e^{\ir t \chi^w(x,D)}$. That is the
    reason why we restrict here to the case of quadratic
    perturbations. In order to deal with more general perturbations
    one needs further ideas which will be developed elsewhere.
\end{remark}

Next we need to know how a time dependent transformation transforms a
classical and a quantum Hamiltonian. Precisely, consider a 1-parameter
family of (Hamiltonian) functions $\chi(t,x,\xi)$ {(where $t$ is thought as an external parameter)} and denote by
$\phi^\tau( t, x,\xi)$ the time $\tau$ flow it generates, precisely the
solution of
\begin{equation}
  \label{chitau}
\frac{dx}{d\tau}=\frac{\partial \chi}{\partial \xi}(t,x,\xi)\ , \quad
\frac{d\xi}{d\tau}=-\frac{\partial \chi}{\partial x}(t,x,\xi) \ .
\end{equation}
Consider the time dependent coordinate transformation
\begin{equation}
  \label{timedep}
(x,\xi)=\phi^1(t,x',\xi'):=\left.\phi^\tau(t,x',\xi')\right|_{\tau=1}\ .
\end{equation}

\begin{remark}
  \label{key.3}
Working in the extended phase space in which time and a new momentum
conjugated to it are added, it is easy to see that the coordinate
transformation \eqref{timedep} transforms a Hamiltonian system with
Hamiltonian $h$ into a Hamiltonian system with Hamiltonian $h'$ given
by
\begin{equation}
  \label{timetras}
h'(t,x',\xi')=h(\phi^1(t,x',\xi'))-\int_0^1  \frac{\partial
  \chi}{\partial t}\left(t,\phi^\tau(t, x',\xi')\right) \, d\tau \ .
\end{equation}
\end{remark}

\begin{remark}
  \label{timetrquantum}
If the operator $\chi^w(t,x,D)$ is selfadjoint for any fixed $t$, then the
transformation
\begin{equation}
  \label{timequan}
\psi=e^{-\im\chi^w(t,x,D) }\psi'
\end{equation}
transform $\im \dot
\psi=H\psi$ into $\im \dot
\psi'=H'\psi'$ with
\begin{equation}
  \label{hqua}
H'=e^{\im\chi^w(t,x,D) }He^{-\im\chi^w(t,x,D)
}-\int_0^1e^{\im \tau \chi^w(t,x,D) } \, \Big(\partial_t \chi^w(t,x,\xi)\Big) \, e^{-\im
  \tau \chi^w(t,x,D) } \, d\tau \ .
\end{equation}
This is seen by an explicit computation. For example see Lemma 3.2 of
\cite{Bam16I}. 
\end{remark}
\noindent So in view of Remark \ref{key.2}, provided that
transformation \eqref{timequan} is well defined in the quadratic case,
the quantum transformed Hamiltonian \eqref{hqua} is the exact
quantization of the transformed classical Hamiltonian
\eqref{timetras}.

To study the analytic properties of the transformation
  \eqref{timequan} we will use the following simplified version of
  Theorem 1.2 of \cite{maro} (to which we refer for the proof). 
\begin{theorem}[\cite{maro}]
\label{thm:maro}
Let $H_0$ be the Hamiltonian of the harmonic oscillator. If $X$ is an
operator symmetric on $\cH^\infty$ such that $ X H_0^{-1} $ and
$[X, H_0] H_0^{-1}$ belong to $\cL(\cH^s)$ for any $s \geq 0$, then
the Schr\"odinger equation
$$
\im \partial_\tau \psi = X \psi  \ 
$$
is globally well posed in $\cH^s$ for any $s$, and its  unitary propagator $e^{-\im \tau X }$ belongs to $\cL(\cH^s)$, $\forall s \geq 0$. Furthermore one has the quantitative estimate
\begin{equation}
\label{q.est}
 c_s\Vert\psi\Vert_s\leq \Vert e^{-\im \tau  X}\psi\Vert_s\leq C_s\Vert\psi\Vert_s,\  \  \forall \tau\in [0,1] \ , 
\end{equation}
where the constants $c_s, C_s>0$ depend only on 
$\norm{[X,  H_0^{s}] H_0^{-s}}_{\cL(\cH^0)}$.
\end{theorem}


The properties of the transformation are
  given by the next lemma and are closely related to the standard
  properties on the smoothness in time of the semigroup generated by an
  unbounded operator.

\begin{lemma}
\label{lem:flow} 
Let $\chi(\rho, x, \xi)$ be a polynomial in $(x, \xi)$ of degree at
most 2 with real coefficients depending in a $C^{\infty}$  way on $\rho
\in \R^n$. Then $\forall \rho\in \R^n$, the operator $\chi^w(\rho, x,D)$ is
selfadjoint in $L^2(\R^d)$. Furthermore $\forall s \geq 0$, $\forall
\tau \in \R$ the following holds true:
\begin{itemize}
\item[(i)] the map $\rho \mapsto e^{-\ir \tau  \chi^w(\rho,x,D)} \in 
 C^0(\R^n, \, \cL(\cH^{s+2},\cH^s)) $.
\item[(ii)] $\forall \psi \in \cH^s$, the map $\rho \mapsto  e^{-\ir \tau  \chi^w(\rho,x,D)}\psi \in 
 C^0(\R^n, \, \cH^s)$.
 \item[(iii)] $\forall r \in \N$   the map $\rho \mapsto e^{-\ir \tau  \chi^w(\rho,x,D)} \in 
 C^r(\R^n, \, \cL(\cH^{s+4r+2},\cH^s)) $.
 \item[(iv)] If the coefficients of $\chi(\rho,x,\xi)$ are  uniformly bounded  in $\rho\in\R^n$  then for any $s>0$ there exist $c_s>0, C_s>0$ such that  we have
 $$
 c_s\Vert\psi\Vert_s\leq \Vert e^{-\ir \tau  \chi^w(\rho,x,D)}\psi\Vert_s\leq C_s\Vert\psi\Vert_s,\ \ \forall \rho\in\R^n, \forall \tau\in [0,1].
 $$
\end{itemize}
\end{lemma}
\begin{proof}
First we remark that in this lemma the quantity $\rho$ plays
  the role of a parameter. Since $\chi(\rho,x,\xi)$ is a real valued
  polynomial in $(x, \xi)$ of degree at most 2, the operator
  $\chi^w(\rho, x, D)$ is selfadjoint in $L^2(\R^d)$, so $\forall \rho \in
  \R^n$ the propagator $ e^{-\im \tau \chi^w(\rho, x, D)}$ is unitary on
  $L^2(\R^d)$. \\ To show that $e^{-\im \tau \chi^w(\rho, x, D)}$ maps
  $\cH^s$ to itself, $\forall s >0$, $\forall \rho \in \R^n$, we apply
  Theorem \ref{thm:maro}. Indeed since $\chi^w(\rho, x, D)$ has a
  polynomial symbol, then $\chi^w(\rho, x, D) H_0^{-1}$ and the
  commutator $[H_0, \chi^w(\rho, x, D)]H_{0}^{-1}$ belong to
  $\cL(\cH^s)$, $\forall s \geq 0$. Item $(iv)$ follows by estimate
  \eqref{q.est} and the remark that $\norm{ [H_0^s, \chi^w(\rho, x,
      D)]H_{0}^{-s}}_{\cL(\cH^0)}$ is bounded uniformly in
  $\rho$.\\ To prove item $(i)$ we use the Duhamel formula 
\begin{equation}
   \label{duhamel}
   e^{-\im \tau B} -e^{-\im \tau A} =  \im \int_0^\tau e^{-\im (\tau-\tau_1)A}\, (A-B)\, e^{-\im \tau_1 B} \, d\tau_1 \ .
   \end{equation}
   Then choosing $B=\chi^w(\rho + \rho',x,D)$, $A =\chi^w(\rho,x,D)$ one has that $\forall \, 0 \leq \tau \leq 1$
   $$
  \norm{  e^{-\im \tau \chi^w(\rho + \rho',x,D) } -e^{-\im \tau \chi^w(\rho,x,D)} }_{\cL(\cH^{s+2}, \cH^s)}
   \leq C \norm{\chi^w(\rho + \rho',x,D) - \chi^w(\rho,x,D)}_{\cL(\cH^{s+2}, \cH^s)} \ . 
   $$
   This proves item $(i)$. 
   Continuity in item $(ii)$ is deduced by $(i)$ with a standard
   density argument.
Finally item $(iii)$ is  proved by induction on $r$ again using  the Duhamel formula \eqref{duhamel}. 
\end{proof}

Remark \ref{key.3}, Remark \ref{timetrquantum} and Lemma \ref{lem:flow} imply  the following important proposition.

\begin{proposition}
  \label{prop:flow}
Let $\chi(t,x,\xi)$ be a polynomial of degree at most 2 in $x$ and
$\xi$ with smooth time dependent coefficients. If the transformation
\eqref{timedep} transforms a classical system  with Hamiltonian $h$ into a
Hamiltonian system with Hamiltonian $h'$, then the transformation
\eqref{timequan} transforms the quantum system with Hamiltonian $h^w$
into the quantum system with Hamiltonian $(h')^w$.
\end{proposition}

{As a consequence, for quadratic Hamiltonians, the quantum KAM
  theorem will follow from the corresponding classical KAM theorem.}

To give the needed result,
consider the classical time dependent Hamiltonian
\begin{equation}
\label{ham.cla}
 h_\epsilon(\omega t,  x,\xi) :=  \sum_{1\leq j\leq d}
 \nu_j\frac{x_j^2+\xi_j^2}{2}+\epsilon\, W(\omega t, x,\xi) \ ,
\end{equation}
with $W$ as in the introduction.  The following KAM theorem holds.
 \begin{theorem}
\label{KAMclassico} Assume that $\nu_j\geq\nu_0>0$ for $j=1,\cdots,d$
and that $\T^n \times \R^{d} \times \R^d\ni(\theta, x,\xi)\mapsto
W(\theta, x,\xi)\in\R$ is a   polynomial in $(x,\xi)$ of degree at
most 2 with coefficients which are real analytic functions of $\theta\in \T^n$.\\ 
Then there exists $\epsilon_*>0$ and $C>0$, such that for
$|\epsilon|<\epsilon_*$ the following holds true:
\begin{itemize}
\item[(i)] there
exists a closed set $\E_\epsilon \subset (0,2\pi)^n$ with $\meas ((0,2\pi)^n\setminus \E_\epsilon )\leq C \eps^{\frac{1}{9}}$ ;
\item[(ii)] for any $\om\in\E_\eps$, there exists an analytic map $\theta\mapsto
 A_{\om}(\theta) \in {\rm sp}(2d) $ (symplectic
 algebra\footnote{{recall that a real $2d\times 2d$ matrix $A$ belongs
   to ${\rm sp}(2d)$ iff $JA$  is symmetric}} of dimension
 $2d$) and an analytic map $\theta\mapsto V_\om(\theta)\in\R^{2d} $,
 such that the change of coordinates
 \begin{equation}
   \label{kamclach}
(x',\xi')=e^{A_{\om}(\omega t)}(x,\xi)+ V_\om(\omega t) 
 \end{equation}
conjugates the Hamiltonian equations of \eqref{ham.cla} to the Hamiltonian equations of a homogeneous polynomial $h_\infty(x,
\xi)$ of degree 2 which is positive definite. Finally both $A_{\om}$
and $V_\om$ are $\epsilon$ close to zero.
\end{itemize}
Furthermore  $h_\infty$ can be diagonalized: there exists a matrix
 $\cP\in {\rm Sp}(2d)$ (symplectic group of dimension $2d$) such that,
  denoting $(y,\eta) =\cP(x,\xi)$ we have
\begin{equation}
\label{fin}h_\infty\circ \cP^{-1}(y,\eta)=\sum_{j=1}^d \nu_j^\infty (y_j^2+\eta_j^2)\end{equation}
where $\nu_j^{\infty}=\nu_j^{\infty}(\om)$ are defined on
$\E_\eps$ and fulfill the estimates 
\begin{equation}
\label{freq:est}
|\nu_j^{\infty}-\nu_j|\leq C\eps,\quad j=1,\cdots,d.
\end{equation}
\end{theorem}

 \begin{remark}
   \label{small}
In general, the matrix $\cP$ is not close to
identity.  
However, in case the frequencies $\nu_j$ are non resonant, then $\cP = \uno$.
 \end{remark}
 
KAM theory in finite dimensions is nowadays standard. In particular we
believe that Theorem \ref{KAMclassico} can be obtained combining the
results of \cite{eliasson88,you}. However, for the reader convenience
and the sake of being self-contained, we add in Section \ref{2} its
proof.

Theorem  \ref{m.1} follows immediately combining the results of 
Theorem  \ref{KAMclassico} and Proposition  \ref{prop:flow}. 
\begin{proof}[Proof of Theorem \ref{m.1}]
We see easily that the change of coordinates \eqref{kamclach} has the form \eqref{timedep}  with an Hamiltonian $\chi_\omega(\omega t,x,\xi)$
which is a polynomial in $(x, \xi)$ of degree at most
2 with real, smooth and  uniformly bounded  coefficients in $t\in\R$.\\
Define $U_\omega(\omega t) = e^{-\im\chi_\omega^w(\omega t, x, D)}$.  
By Proposition \ref{prop:flow} it conjugates the original equation \eqref{schro}  to \eqref{rido}  where $H_\infty := {\rm Op}^w(h_\infty)$. 
\\
{Furthermore $\theta \mapsto U_\omega(\theta)$  fulfills
  $(i)$--$(iv)$ of Lemma \ref{lem:flow}, from which it follows immediately that $\theta \mapsto U_\omega(\theta)$ fulfills 
item $(i)$, $(iii)$ of Theorem \ref{m.1}. Concerning item $(ii)$, by Taylor formula the quantity $\norm{\uno - U_\omega(\theta)}_{\cL(\cH^{s+2}, \cH^s)}$ is controlled by $\norm{\chi_\omega^w(\theta, x, D)}_{\cL(\cH^{s+2}, \cH^s)}$, from which  estimate \eqref{pro1} follows.}

Finally using the metaplectic representation (see \cite{CombRob}) and \eqref{fin}, there exists a unitary transformation in $L^2$,  $\cR(\cP^{-1})$, such that
$$\cR(\cP^{-1})^* H_\infty \cR(\cP^{-1})= \sum_{j=1}^d \nu_j^\infty (x_j^2 - \partial_{x_j}^2) \ .$$
\end{proof}
We prove now Corollary \ref{cor.1}.

\begin{proof}[Proof of Corollary \ref{cor.1}]
Consider first the propagator $e^{- \im t H_\infty}$. We claim that 
\begin{equation}
\label{H.inf.born}
\sup_{t \in \R} \norm{ e^{- \im t H_\infty} }_{\cL(\cH^s)} < \infty \ , \qquad \forall t \in \R \ .
\end{equation}
Recall that $H_\infty = h_\infty^w(x, D)$ where $h_\infty(x,\xi)$ is a positive definite symmetric form which can be diagonalized by a symplectic matrix $\cP$.
Since $h_\infty$ is positive definite, there exist $c_0, c_1, c_2 >0$ s.t.
$$
c_1 h_0(x,\xi) \leq c_0 + h_\infty(x, \xi) \leq c_2 (1+ h_0(x, \xi)) \ ,
$$
which implies that 
$C_1 H_0 \leq C_0 + H_\infty \leq  C_2 (1+ H_0)$ as bilinear form. Thus one has the equivalence of norms 
$$
C_s^{-1}\norm{\psi}_{\cH^s} \leq \norm{(H_\infty)^{s/2} \psi}_{L^2} \leq C_s \norm{\psi}_{\cH^s} \ .
$$
Then
$$
\norm{e^{-\im t H_\infty}\psi_0}_{\cH^s} \leq C_s\norm{ (H_\infty)^{s/2}\, e^{-\im t H_\infty} \psi_0}_{L^2} = C_s
\norm{(H_\infty)^{s/2}\,\psi_0}_{L^2} \leq C'_s\norm{\psi_0}_{\cH^s}
$$
which implies \eqref{H.inf.born}.\\
Now let $\psi(t)$ be a solution of \eqref{schro}. By  formula \eqref{decom},  $\psi(t) = U^*_\om(\omega t) e^{- \im t H_\infty}U_\omega(0) \psi_0$. Then   the upper bound in \eqref{unest} follows easily from \eqref{H.inf.born} and $\sup_{t}\norm{ U_\om(\omega t)}_{\cL(\cH^s)}< \infty$, which is a consequence of Lemma \ref{lem:flow}. The lower-bound follows by applying Lemma  \ref{lem:flow} $(iv)$.

Finally estimate \eqref{unest2} follows from \eqref{pro1}.

 \end{proof}

\section{A classical  KAM result.}\label{2}

In this section we prove  Theorem \ref{KAMclassico}. 
We prefer to work in the extended phase space in which we add the angles
$\theta\in\T^n$ as new variables and their conjugated momenta
$I\in\R^n$. Furthermore we will use complex variables defined by
$$
z_j=\frac{\xi_j-\im x_j}{\sqrt 2}\ ,
$$
so that our phase space will be $\T^n\times \R^n\times \C^d$, with
$\C^d$ considered as a real vector space. The symplectic form is $dI\wedge d\theta+\im dz\wedge d\bar z
$ and the Hamilton equations of a Hamiltonian function
$h(\theta,I,z,\bar z)$ are
$$
\dot I=-\frac{\partial h}{\partial \theta}\ ,\quad \dot
\theta=\frac{\partial h}{\partial I}\ ,\quad\dot
z=-\im\frac{\partial h}{\partial \bar z}\ .
$$
In this framework $h_0$ takes the form
$h_0=\sum_{j=1}^{d}\nu_jz_j\bar z_j$ and $W$ takes the form of
polynomial in $z,\bar z$ of degree two
$W(\theta,x,\xi)=q(\theta,z,\bar z)$. The Hamiltonian system associated with the time dependent Hamiltonian
$h_\epsilon$ (see \eqref{ham.cla}) is then equivalent to the Hamiltonian system associated with the time
independent Hamiltonian $\omega \cdot I + h_\epsilon$ (written in
complex variables) in the extended phase space.

\subsection{General strategy}
Let $h$ be a   Hamiltonian in normal form:
\begin{equation}\label{h}
h(I,\theta,z,\bar z)= \om\cdot I +\langle z, N(\om)\bar z\rangle\end{equation}
with $N\in\M_H$ the set of Hermitian matrix. Notice that at the beginning of the procedure $N$ is diagonal, 
$$N=N_0=\diag(\nu_j,\ j=1,\cdots,d)$$
 and is independent of $\om$. Let
 $q\equiv q_\om$ be a polynomial Hamiltonian which takes real values: $q(\theta,z,\bar z)\in\R$ for $\theta\in\T^n$ and $z\in\C^d$. We write
\begin{equation}
\label{q}
q(\theta,z, \bar z)=\lan z,Q_{zz}(\theta) z\ran + \lan z,Q_{z\bar z}(\theta)\bar z\ran +\lan \bar z,\bar Q_{zz}(\theta)\bar z\ran + \lan Q_z(\theta),z\ran+\lan  \bar Q_{\bar z}(\theta), \bar z\ran
\end{equation}
where $Q_{zz}(\theta)\equiv Q_{zz}(\om,\theta)$ and $Q_{z\bar z}(\theta)\equiv Q_{z\bar z}(\om,\theta)$ are $d\times d$ complex matrices and $Q_{z}(\theta)\equiv Q_{z}(\theta,\om)$ is a vector in $\C^d$. They all depend analytically on the angle $\theta\in\T^n_\s:=\{x+iy\mid x\in\T^n,\ y \in\R^n,\ |y|<\s\}$. We notice that $Q_{z\bar z}$ is Hermitian while $Q_{zz}$ is symmetric. The size of such polynomial function depending analytically on $\theta\in\T^n_\s$ and $C^1$ on
$\om\in\D=(0,2\pi)^n$ will be controlled by the norm
 $$
 [q]_{\s}:=\sup_{\substack{|\Im \theta|<\s \\ \om\in\D,\ j=0,1}}  \|\p_\om^jQ_{zz}(\om,\theta)\|+\sup_{\substack{|\Im \theta|<\s \\ \om\in\D,\ j=0,1}} \|\p_\om^jQ_{z\bar z}(\om,\theta)\| +  \sup_{\substack{|\Im \theta|<\s \\ \om\in\D,\ j=0,1}} |\p_\om^jQ_z(\om,\theta)|
 $$
and we denote by $\mathcal Q(\s)$ the  class of Hamiltonians of the form \eqref{q} whose norm $[\cdot ]_{\sigma}$ is finite.\\
Let us assume that $[q]_{\s}=\cO (\eps)$. We search for $\chi\equiv \chi_\om\in \mathcal Q(\s)$
with $[\chi]_{\sigma}=\cO (\eps)$ such that its  time-one flow 
$\phi_\chi\equiv \phi_\chi^{t=1}$ (in the extended phase space, of
course) transforms the Hamiltonian $h+ q$ into
\begin{equation} \label{eq}
(h+ q(\theta))\circ \phi_\chi=h_++ q_+(\theta),\quad \om\in\D_+
\end{equation}
where $h_+=\om\cdot I +\langle z, N_+\bar z \rangle$ is a new normal form, $\eps$-close to $h$,  the new perturbation $q_+\in\mathcal Q(\s_)$  is of size\footnote{Formally we could expect  $q_+$ to be of size $O(\eps^2)$ but the small divisors and the reduction of the analyticity domain will lead to an estimate of the type $\cO (\eps^{\frac 32})$. } $\cO (\eps^{\frac 32})$ and $\D_+\subset \D$ is $\eps^\a$-close to $\D$ for some $\a>0$. Notice that all the functions are defined on the whole open set $\D$ but the equalities \eqref{eq} holds only on $\D_+$ a subset of $\D$ from which we excised the "resonant parts".\\
As a consequence of the Hamiltonian structure we have  that
$$
(h+ q(\theta))\circ \phi_\chi= h+\{ h,\chi \}+q(\theta)+ \cO (\eps^{\frac 32}),\quad \om\in\D_+ \ .
$$
So to achieve the goal above 
we should  solve the {\it homological equation}:
\begin{equation} \label{eq-hom}
\{ h,\chi \}= h_+-h -q(\theta)+\cO (\eps^{\frac 32}),\quad \om\in\D_+.
\end{equation}
Repeating iteratively 
the same procedure with $h_+$ instead of $h$, we will construct a change of variable $\phi$ such that
$$
(h+ q(\theta))\circ \phi=\om\cdot I +h_\infty, \quad \om\in\D_\infty\,,
$$
with $h_\infty=\langle z, N_\infty(\om)\bar z\rangle$ in normal form and $\D_\infty$ a $\eps^\a$-close subset of $\D$.
 Note that  we will be forced to solve the homological equation not
only for the diagonal normal form $N_0$, but for 
more general normal form Hamiltonians
\eqref{h} with  $N$ close to $N_0$ .

\subsection{Homological equation}\label{section-homo}

\begin{proposition}\label{prop:homo1}
Let $\D=(0,2\pi)^n$ and  $\D\ni\om\mapsto N(\om)\in \M_H$ be a $C^1$ mapping that  verifies 
\begin{equation}\label{ass}
 \left\| \p_\om^j (N(\om)-N_0) \right\| < \frac {\min(1,\nu_0)}{\max(4,d)}\end{equation}
for $j=0,1$ and $\om\in \D$.
Let $h=\om\cdot I+\lan z,N\bar z\ran$, $q\in\mathcal Q(\s)$ , $\ka>0$ and $K\ge 1$. \\
 Then there exists a closed subset $\D'=\D'( \ka,K)\subset \D$, satisfying 
 \begin{equation}\label{estim:D}\meas (\D\setminus \D')\leq  
 CK^{n}\ka, \end{equation}
and there exist $\chi, r \in\cap_{0\leq\s'<\s}\mathcal Q(\s')$ and $\D\ni\om\mapsto \tilde N(\om)\in \M_H$ a $C^1$ mapping such that for all $\om\in\D'$
\begin{equation}\label{ho}
\{h,\chi \}+q=\lan z,\tilde N\bar z\ran +r\ .
\end{equation} 
Furthermore for all $\om\in\D$
\begin{equation}\label{B}
 \left\|\p_\om^j \tilde N(\om)\right\|\leq   [q]_{\s}\ , \qquad j=0,1\end{equation}
 and for all $0\leq \s'<\s$ 
 \begin{align}
\label{estim-homoR}
[r]_{\s'}&\leq  C\ \frac{e^{-\frac12 (\s-\s')K}}{ (\s-\s')^{n}}
[q]_{\s}\,,\\
\label{estim-homoS}
[\chi]_{\s'}&\leq \frac{CK}{\ka^2 (\s-\s')^{n}}
[q]_{\s}\,.
 \end{align} 

\end{proposition}
\proof
Writing the Hamiltonians $h$, $q$ and $\chi$ as in \eqref{q}, the homological equation \eqref{ho} is equivalent to the three following equations (we use that $N$ is Hermitian, thus $\bar N={}^tN$):
\begin{equation}\label{homo1}
\om\cdot\nabla_\theta X_{z\bar z} - \im [N,X_{z\bar z}]= \tilde N-Q_{z\bar z}+ R_{z\bar z},
\end{equation}
\begin{equation}\label{homo2}
\om\cdot\nabla_\theta X_{z z} - \im (NX_{z z}+X_{z z}\bar N)= -Q_{z z}+ R_{z z},
\end{equation}
\begin{equation}\label{homo3}
\om\cdot\nabla_\theta X_z + \im NX_z= -Q_z+R_z\, .
\end{equation}
First we solve \eqref{homo1}. To simplify notations we drop the indices $z\bar z$.
Written in Fourier variables (w.r.t. $\theta$), \eqref{homo1} reads
\begin{equation}\label{homo+} 
 \im \om\cdot k\  \hat X_k-\im [N,\hat X_k]=\delta_{k,0}\tilde N- \hat Q_k+\hat R_k,\quad k\in\Z^n
 \end{equation}
where  $\delta_{k,j}$ denotes the Kronecker symbol. \\
When $k=0$ we solve this equation   by defining
 $$
 \hat X_0=0,\quad \hat R_0=0 \ \  \text{ and } \ \ \tilde N=\hat Q_0.
$$
We notice that $\tilde N\in\M_H$ and satisfies \eqref{B}.\\
When $|k|\geq K$ equation \eqref{homo+} is solved by defining
\begin{align}
\label{klarge}
\hat R_k=\hat Q_k, \quad \hat X_k=0 \text{ for }|k|\geq K.
\end{align}
Then we set
$$\hat R_k=0\quad \text{ for }|k|\leq K
$$
in such a way that $r \in\cap_{0\leq\s'<\s}\mathcal Q(\s')$ and by a standard argument $r$ satisfies \eqref{estim-homoR}.  Now it remains to solve the equations for $\hat X_k$, $0<|k|\leq K$ which we rewrite as
\begin{equation}\label{L_k}L_k(\om) \hat X_k= i\hat Q_k
\end{equation}
where $L_k(\om)$ is the linear operator from $\M_S$, the space of symmetric matrices, into itself defined by
$$L_k(\om):  M \mapsto k\cdot \om-[N(\om),M] \ . $$
We notice that $\M_S$ can be endowed with the Hermitian product: $(A,B)=Tr(\bar A B)$ associated with the Hilbert Schmidt norm. Since $N$ is Hermitian, $L_k(\om)$ is self adjoint for this structure. As a first consequence we get
\begin{equation}\label{L-1}\|(L_k(\om))^{-1}\|\leq \frac1{\min\{| \lambda|,\, \lambda \in\varSigma(L_k(\om))\}}=\frac1{\min\{|k\cdot \om-\a(\om)+\beta(\om)| \mid \a, \beta \in\varSigma(N(\om))\}}\end{equation}
where for any matrix $A$, we denote its spectrum by $\varSigma (A)$.\\
Let us recall an important result of perturbation theory which is a consequence of  Theorem 1.10 in \cite{Kato} (since hermitian matrices are normal matrices):
\begin{theorem}[\cite{Kato} Theorem 1.10]
Let $I\subset \R$ and $I\ni z\mapsto M(z)$ a holomorphic curve of hermitian matrices.  Then all the eigenvalues  and associated eigenvectors of $M(z)$ can be parametrized  holomorphically on $I$.
\end{theorem}
Let us assume for a while that $N$ depends analytically of $\om$ in
such a way that $\om\mapsto L_k(\om)$ is analytic.  Fix a direction
$z_k\in\R^n$, the eigenvalue $\lambda_k (\om)=k\cdot\om
-\a(\om)+\beta(\om)$ of $L_k(\om)$ is $C^1$ in the
direction\footnote{i.e. $t\mapsto \lambda_k (\om+t z_k))$ is a
  holomorphic curve on a neighborhood of $0$, we denote $\partial_\om
  \lambda (\om)\cdot z_k$ its derivative at $t=0$.} $z_k$ and the
associated unitary eigenvector, denoted by $v(\om)$, is also piece wise $C^1$ in the direction $z_k$.  Then, as a consequence of the
hermiticity of $L_k(\om)$ we have
$$\partial_\om \lambda (\om)\cdot z_k= \langle v(\om),(\partial_\om L_k(\om)\cdot z_k)\ v(\om)\rangle.$$
Therefore, if $N$ depends analytically of $\om$, we deduce using \eqref{ass} and choosing $z_k=\frac k{|k|}$
\begin{equation}\label{dom}\left|  \partial_\om \lambda_k (\om)\cdot
  \frac k{|k|}\right|\geq |k|-2\| \partial_\om N\|\geq \frac 12\,
  \quad \text{for } k\neq 0,\end{equation}
which extends also to the points of discontinuity of $v(\omega)$.
Now given a matrix $L$ depending on the parameter $\om\in\D$, we define
$$\D(L,\ka)=\{\om\in\D\mid \|L(\om)^{-1}\|\leq \ka^{-1}\}$$
and we recall the following classical lemma:
\begin{lemma}
\label{lem-mes}
Let $f:[0,1]\mapsto\R$ a $C^1$-map satisfying $|f'(x)|\geq \delta$ for all $x\in[0,1]$ and let $\ka>0$.  Then
$$\meas\{x\in[0,1]\mid |f(x)|\leq \ka\}\leq \frac{\ka}{\delta}.$$
\end{lemma}
Combining this Lemma \eqref{L-1} and \eqref{dom} we deduce that, if $N$
depends analytically of $\om$, then for $k\neq 0$
\begin{equation}\label{meask}\meas (\D\setminus \D(L_k,\ka))\leq C\ka\end{equation}
Now it turns out that, by a density argument,  this last estimate remains valid (with a larger constant $C$) when $N$ is only a $C^1$ function of $\om$ : the point is that \eqref{dom} holds true uniformly for close analytic approximations of $N$.
\\ 
 In particular defining $$\D'=\bigcap_{0<|k|\leq K}\D(L_k,\ka)$$
 $\D'$ is closed and satisfies \eqref{estim:D}.\\
By construction, $\hat X_k(\om):=iL_k(\om)^{-1}\hat Q_k$ satisfies \eqref{L_k} for $0<|k|\leq K$ and $\om\in\D(L_k,\ka)$ and
  \begin{equation}\label{Xk}\|\hat X_k(\om)\|\leq \ka^{-1}\|\hat Q_k(\om)\|,\quad \om\in\D(L_k,\ka).\end{equation}
  
 It remains  to extend $\hat X_k(\cdot)$ on $\D$.
Using again \eqref{ass} we have for any $|k|\leq K$ and any unit vector $z$, $|\partial_\om \lambda (\om)\cdot z|\leq CK$. Therefore 
$$\dist ( \D\setminus \D(L_k,\ka), \D(L_k,\ka/2))\geq \frac{\ka}{CK}$$ 
and we can construct (by a convolution argument) for each $k$, $0<|k|\leq K$, a $C^1$ function $g_k$ on $\D$ with 
\begin{equation}\label{g}|g_k|_{C^0(\D)}\leq C,\quad|g_k|_{C^1(\D)}\leq CK\ka^{-1}\end{equation} (the constant C is independent of k) and such that
$g_k(\om)=1$ for $\om\notin \D(L_k,\ka)$ and $g_k(\om)=0$ for $\om\in \D(L_k,\ka/2)$. Then $
\tilde X_k=g_k\hat X_k$ is a $C^1$ extension of $\hat X_k$ to $\D$. Similarly we define $\tilde Q_k,=g_k\hat Q_k$ in such a way that $\tilde X_k$ satisfies 
$$L_k(\om)\tilde X_k(\om)=i\tilde Q_k(\om),\quad 0<|k|\leq K,\ \om\in\D.$$
Differentiating with respect to $\om$
 leads to 
\begin{equation*}
 L_k(\om)  \partial_{\om_j}\hat X(k)=\im \partial_{\om_j} \hat Q(k)- k_j\hat X(k)+ [\partial_{\om_j}N,\hat X(k)],\quad 1\leq j\leq n\ . \end{equation*}
Denoting $B_k(\om)=\im \partial_{\om_j}\tilde Q_k(\om)- k_j\tilde X_k(\om)+ [\partial_{\om_j}N(\om),\tilde X_k(\om)]$ we have
$$\left\| \partial_{\om_j}\tilde X_k(\om)\right\|\leq 
\ka^{-1}
 \|B_k(\om)\| ,\quad \om\in\D \ .
$$
Using \eqref{ass}, \eqref{Xk} and \eqref{g} we get for $|k|\leq K$ and $\om\in\D$
\begin{align*}\|B_k(\om)\| &\leq \| \partial_{\om_j}\tilde Q_k(\om)\|+K\|\tilde  X_k(\om)\|+2 \|\partial_{\om_j}N(\om)\| \| \tilde X_k(\om)\|\\
&\leq CK\ka^{-1} (\| \partial_{\om_j}\hat Q(k,\om)\|+\| \hat Q(k,\om)\|).\end{align*}
Combining the last two estimates we get 
$$\sup_{\om\in\D,\ j=0,1}\left\| \partial^j_{\om}\tilde X_k(\om)\right\|\leq CK\ka^{-2}\sup_{\om\in\D,\ j=0,1}\left\| \partial^j_{\om}\hat Q_k(\om)\right\|.$$
Thus defining $$X_{z\bar z}(\om,\theta)=\sum_{0<|k|\leq K}\tilde X_k(\om)e^{ik\cdot \theta}$$
$X_{z\bar z}(\om,\cdot)$ satisfies \eqref{homo1} for $\om\in\D'$ and leads to
\eqref{estim-homoS} for $\chi_{z\bar z}(\om,\theta,z,\bar z)=\langle z,X_{z\bar z}(\om,\cdot) \bar z\rangle$ . 

\smallskip

We solve \eqref{homo3} in a similar way. We notice that in this case we face the small divisors  $|\om\cdot k-\alpha(\om)|,\  k\in\Z^n$ where $\a\in\varSigma (N(\om))$. In particular for $k=0$  these quantities are $\geq \frac{\nu_0}2$ since $|\a-\nu_j|\leq \frac {\nu_0} 4$ for some $1\leq j\leq d$ by \eqref{ass}.

Written in Fourier and dropping indices $zz$ \eqref{homo2} reads
\begin{equation}\label{hom2k}
 \im \om\cdot k\  \hat X(k)-\im(N\hat X(k)+\hat X(k)\bar N)=- \hat Q(k)+\hat R(k)\, .
\end{equation}
So to mimic the resolution of \eqref{homo+} we have to replace  the operator $L_k(\om)$ by the operator $M_k(\om)$ defined on $\M_S$ by
$$M_k(\om)X:=\om\cdot k+NX+X\bar N.$$
This operator is still self adjoint for the Hermitian product $(A,B)=Tr(\bar A B)$ so the same strategy apply. Nevertheless we have to consider differently the case $k=0$. In that case we use that the eigenvalues of $M_0(\om)$ are  close to eigenvalues of the operator $M_0$ defined by
$$M_0:  X \mapsto N_0X+X\bar N_0=N_0X+XN_0$$
with $N_0=\diag(\nu_j,\ j=1,\cdots,d)$ a real and diagonal matrix. Actually in view of \eqref{ass}
$$\|(L-L_0)M\|_{HS}\leq \| N-N_0\|_{HS}\, \|M\|_{HS}\leq d \| N-N_0\|\, \|M\|_{HS}\leq  {\nu_0}.$$
 The eigenvalues of $L_0$ are $\{\nu_j+\nu_\ell\mid j,\ell=1,\cdots,d\}$ and they are all larger than $2\nu_0$.  We conclude that all the eigenvalues of $M_0(\om)$ satisfy $|\a (\om)|\geq {\nu_0}$. The end of the proof follow as before.

\endproof

\subsection{The KAM step.}
 Theorem \ref{KAMclassico} is proved by an iterative KAM procedure. We begin with the initial Hamiltonian $h_0+q_0$ where
\begin{equation}\label{h0}
h_0(I,\theta,z,\bar z)= \om\cdot I +\langle z, N_0\bar z\rangle\,,\end{equation}
 $N_0=\diag (\nu_j,\ j=1,\cdots,d)$, $\om\in\D \equiv [1,2]^n$ and the quadratic perturbation $q_0=\eps W\in \mathcal Q(\s,\D)$
 for some $\s>0$.  Then we construct iteratively the change of variables $\phi_m$, the normal form $h_m=\om\cdot I +\langle z, N_m\bar z\rangle$ and the perturbation $q_m\in \mathcal Q(\s_m,\D_m)$  as follows: assume that the construction is done up to step $m\geq0$ then
\begin{itemize}
\item[(i)]
Using Proposition \ref{prop:homo1}  we construct $\chi_{m+1}$, $r_{m+1}$ and $\tilde N_{m}$ the solution of the homological equation:
\begin{equation} \label{eq-homo}
\{ h,\chi_{m+1} \}=\langle z,  \tilde N_{m}\bar z\rangle -q_m(\theta)+r_{m+1},\quad \om\in\D_{m+1},\ \theta\in\T^n_{\s_{m+1}} .
\end{equation}
\item[(ii)] We define $h_{m+1}:=\om\cdot I+\langle z,  N_{m+1}\bar z\rangle$ by
\begin{equation}\label{Nm} N_{m+1}=N_m+\tilde N_m\,,\end{equation}
and
\begin{equation}\label{qm}  q_{m+1} :=r_{m}  +\int_0^1\{(1-t)(h_{m+1}-h_m+r_{m+1})+tq_{m}, \chi_{m+1}\}\circ \phi_{\chi_{m+1}}^t\dd t\,.
\end{equation}
\end{itemize}
By construction, if $Q_{m}$ and $N_{m}$ are Hermitian, so are $R_{m}$ and  $S_{m+1}$ by the resolution of the homological equation, and  also $N_{m+1}$ and $Q_{m+1}$. \\
For any regular Hamiltonian $f$ we have, using the Taylor expansion of $f\circ\phi_{\chi_{m+1}}^t$ between $t=0$ and $t=1$
$$f\circ  \phi^1_{\chi_{m+1}}=f+\{f,\chi_{m+1}\}+\int_0^1 (1-t)\{\{f,\chi_{m+1}\},\chi_{m+1}\}\circ \phi_{\chi_{m+1}}^t\dd t\,.$$
Therefore we get for $\om\in\D_{m+1}$ 
\begin{align*}
(h_m+q_m)\circ  \phi^1_{\chi_{m+1}}=h_{m+1}+q_{m+1}.
\end{align*}

\subsection{Iterative lemma}
Following the general scheme above we have
$$(h_0+q_0)\circ \phi^1_{\chi_{1}}\circ\cdots\circ \phi^1_{\chi_m}= h_{m}+q_{m}$$
where  $q_m$ is a polynomial of degree two and  $h_m=\om\cdot I +\lan z,N_m\bar z \ran$ with $N_m$ a Hermitian matrix. At step $m$  the Fourier series are truncated at order $K_m$ and the small divisors are controlled by $\ka_m$. Now we specify the choice of all the parameters for $m\geq 0$ in term of $\eps_m$ which will control  $[q_m]_{\D_m,\s_m}$. \\
First we 
define  $\eps_0=\eps$, $\s_0=\s$, $\D_0 = \D$ and for $m\geq 1$ we choose
\begin{align*}
\s_{m-1}-\s_m=&C_* \s_0 m^{-2},\qquad 
K_m=2(\s_{m-1}-\s_m)^{-1}\ln \eps_{m-1}^{-1},\qquad
\ka_{m}=\eps_{m-1}^{\frac 18} 
\end{align*}
where $(C_*)^{-1} =2\sum_{j\geq 1}\frac 1{j^2}$.\\

\begin{lemma}\label{iterative} There exists  $\eps_*>0$ depending on  $d$, $n$ such that, for  $|\eps|\leq\eps_*$ and
$$
\eps_{m}= \eps^{(3/2)^m}\ , \quad m\geq 0\,,
$$
 we have the following:\\
For all $m\geq 1$ there exist closed subsets $\D_m\subset\D_{m-1}$,
 $h_m=\om\cdot I  +\lan z,N_m \bar z \ran$  in normal form where $\D_m \ni \omega \mapsto N_m(\om) \in \M_H \in C^1$   and there exist
$\chi_m,q_m\in \mathcal Q(\D_m,\s_m)$ such that for $m\geq1$
\begin{itemize}
\item[(i)]  The symplectomorphism \begin{equation} \label{Phik} \phi_{m}\equiv\phi_{\chi_m}(\om) \ :\ \R^n\times\T^n\times\C^{2d} \to \R^n\times\T^n\times\C^{2d}, \quad \om\in \D_{m}\end{equation}
is an affine transformation in $(z,\bar z)$, analytic in $\theta\in\T^n_{\s_m}$ and $C^1$ in $\om\in\D_m$ of the form
\begin{equation}
\label{phi_m}
\phi_m(I,\theta,z,\bar z)=( g_m(I,\theta,z,\bar z),\theta, \Psi_m(\theta,z,\bar z)) \ , 
\end{equation}
where, for each $\theta\in\T^n$, $ (z,\bar z)\mapsto \Psi_m(\theta,z,\bar z)$ is a symplectic change of variable on $\C^{2n}$.\\
The map $\phi_m$ links   the Hamiltonian at step $m-1$ and the Hamiltonian at step  $m$, i.e.
$$
(h_{m-1}+q_{m-1})\circ \phi_{m}= h_m+q_m, \quad \forall \om\in\D_m \ .
$$
\item[(ii)]We have the estimates
\begin{align}
\label{DD}\meas(\D_{m-1}\setminus \D_{m})&\leq \eps_{m-1}^{\frac 19},\\
\label{NN}[\tilde N_{m-1}]_{s,\b}^{\D_m}&\leq \eps_{m-1},\\
\label{QQ}[q_m]_{s,\b}^{\D_m,\s_m}&\leq \eps_m,\\
\label{FiFi}\| \phi_m(\om)-\uno\|_{\L( \R^n\times\T^n\times\C^{2d})}&\leq C \eps_{m-1}^{\frac 12},\ \quad \forall   \om\in\D_m.
\end{align}
\end{itemize}
 \end{lemma}
 \proof At step 1, $h_0=\om\cdot I +\lan z,N_0\bar z\ran$ and thus hypothesis \eqref{ass} is trivially satisfied and we can apply Proposition \ref{prop:homo1} to construct $\chi_1$, $N_1$, $r_1$ and $\D_1$ such that for $\om\in\D_1$
 $$\{h_0,\chi_1\}= \lan z,(N_1-N_0)\bar z\ran- q_0+r_1.$$
Then, using \eqref{estim:D}, we have
$$\meas(\D\setminus\D_1)\leq C K_1^{n}\ka_1\leq  \eps_0^{\frac 19}$$
for $\eps=\eps_0$ small enough. 
Using \eqref{estim-homoS} we have for $\eps_0$ small enough
$$
[\chi_1]_{\D_1,\s_1}\leq C \frac{K_1}{\ka_1^{2}(\s_0-\s_1)^n}\eps_0\leq  \eps_0^{\frac12}.
$$
Similarly using \eqref{estim-homoR}, \eqref{B} we have
$$\|N_1-N_0\|\leq  \eps_0,$$ and
$$
[r_1]_{\D_1,\s_1}\leq  C\frac{\eps_0^{\frac {15}8}}{(\s_1-\s_0)^n}\leq\eps_0^{\frac {7}4}
$$
for $\eps=\eps_0$ small enough.
In particular we deduce $\| \phi_1-\uno\|_{\L( \R^n\times\T^n\times\C^{2d})}\leq  \eps_0^{\frac 12}.$
 Thus using \eqref{qm} we get
for $\eps_0$ small enough
$$[q_1]_{\D_1,\s_1}\leq \eps_0^{3/2}=\eps_1.$$
 The form of the flow \eqref{phi_m} follows  since $\chi_1$ is a Hamiltonian of the form \eqref{q}.
  \medskip
 
 Now assume that we have verified Lemma \ref{iterative}  up to step $m$.
We want to perform the step $m+1$. We have $h_m=\om\cdot I +\lan z,N_m\bar z\ran$ and since 
$$ \|N_m-N_0\|\leq \|N_m-N_0\|+\cdots+ \|N_1-N_0\|\leq \sum_{j=0}^{m-1}\eps_j\leq 2\eps_0,$$
hypothesis \eqref{ass} is satisfied and we can apply Proposition \ref{prop:homo1}  to construct $\D_{m+1}$, $\chi_{m+1}$ and $q_{m+1}$. Estimates \eqref{DD}-\eqref{FiFi} at step $m+1$  are proved as we have proved the corresponding estimates  at step 1.
\endproof

\subsection{Transition to the limit and proof of Theorem \ref{KAMclassico}}

Let $\cE_\epsilon=\cap_{m\geq 0}\D_m.$ In view of \eqref{DD}, this is a closed set satisfying
$$\meas(\D\setminus\cE_\epsilon)\leq \sum_{m\geq 0} \eps_m ^{\frac 19}\leq 2 \eps_0^{\frac 19}.$$
Let us denote $\widetilde\phi_N=\phi_{1}\circ\cdots\circ \phi_N$. Due  to \eqref{FiFi} it satisfies for $M\leq N$ and for $\om\in\cE-\eps$
$$\| \widetilde\phi_N- \widetilde\phi_M\|_{\L( \R^n\times\T^n\times\C^{2d})}\leq  \sum_{m=M}^N\eps_m^{\frac 12}\leq 2\eps_M^{\frac 12}\,.$$ 
Therefore  $(\widetilde\phi_N)_N$ is a Cauchy sequence in $\L( \R^n\times\T^n\times\C^{2d})$. 
Thus when $N\to \infty$ the mappings  $\widetilde\phi_N$ converge to a limit mapping $\phi_\infty\in\L( \R^n\times\T^n\times\C^{2d}).$ 
Furthermore since the convergence is uniform on $\om\in\cE_\eps$ and  $\theta\in\T_{\s/2}$, $\phi_\infty^1$ depends analytically on $\theta$ and $C^1$ in  $\om$. Moreover, 
\begin{equation} \label{estim-Phiinf}\|\phi_\infty-\uno\|_{\L( \R^n\times\T^n\times\C^{2d})} \leq \eps_{0}^{\frac 12}\,.\end{equation}
By construction, the map $\widetilde\phi_m$ transforms the original Hamiltonian $h_0+q_0$ into $h_m+q _m$.
When $m\to\infty$, by \eqref{QQ} we get $q_m\to 0$  and by \eqref{NN}  we get $N_m\to N$  where \begin{align}\label{Nom}
N\equiv N(\om) = N_{0} + \sum_{k=1}^{+\infty}\tilde{N}_{k}
\end{align}
is a Hermitian matrix which is $C^1$ with respect to $\om\in\cE_\eps$. 
Denoting  $h_\infty(z,\bar z)= \omega \cdot I + \lan z,N(\om)\bar z \ran$ we have proved 
\begin{equation}
\label{final3}
(h + q(\theta)) \circ \phi_\infty =  h_\infty \ . 
\end{equation}
Furthermore  $\forall \om\in\cE_\eps$ we have, using \eqref{NN}, 
$$\norma{N(\om)-N_0}\leq \sum_{m=0}^\infty \eps_m\leq 2 \eps $$
and thus the eigenvalues of $N(\om)$, denoted $\nu_j^\infty(\om)$ satisfy \eqref{freq:est}.

It remains to explicit the affine symplectomorphism $\phi_\infty$.
At each step of the KAM procedure we have by Lemma \ref{iterative} 
 \begin{equation*}
\phi_m(I,\theta,z,\bar z)=( g_m(I,\theta,z,\bar z),\theta, \Psi_m(\theta,z,\bar z))
\end{equation*}
and therefore
\begin{equation*}
\phi_\infty(I,\theta,z,\bar z)=( g(I,\theta,z,\bar z),\theta, \Psi(\theta, z,\bar z))
\end{equation*}
where  $\Psi(\theta,z,\bar z)=\lim_{m\to\infty}\Psi_1\circ \Psi_2 \circ \cdots \circ \Psi_m$.

It is useful to go back to real variables $(x,\xi)$. More precisely  write each Hamiltonian $\chi_m$ constructed in the KAM iteration in the variables $(x,\xi)$:
\begin{equation}
\label{matrix}
\chi_m(\theta, x, \xi)=\frac{1}{2} \left(\begin{matrix}
x \\ \xi
\end{matrix}\right)\cdot E \, B_m(\theta) \left(\begin{matrix}
x \\ \xi
\end{matrix}\right) + U_m(\theta) \ , 
\qquad
E:=\left[\begin{matrix} 
0 & - \uno \\
\uno & 0 \end{matrix}\right] \ , 
\end{equation}
where   $B_m(\theta)$ is a skew-symmetric matrix of dimension $2d\times 2d$   and $U_m(\theta) \in\R^{2d}$, and they are both of size  $\eps_m$. Then $\Psi_m$ written in the real variables  has the form 
\begin{equation}
\label{ham.flow.r}
\Psi_m (\theta, x,\xi) = e^{ B_m(\theta)}(x,\xi) + T_m(\theta) \ , \qquad {\rm where } \ \  
T_m(\theta) := \int_0^1e^{(1-s)JB_m(\theta)}U_m(\theta) ds \ . 
\end{equation}
\begin{lemma}
\label{stitras}
There exists a sequence of  Hamiltonian matrices $A_l(\theta)$  and vectors $V_l(\theta)\in\R^{2d}$   such that 
\begin{equation}
\label{cano}
\Psi_{1}\circ...\circ \Psi_{l} (x, \xi) =e^{A_l(\theta)}(x,\xi) + V_l(\theta) \ \ \  \forall (x, \xi) \in\R^{2d} \ . 
\end{equation} 
Furthermore, there exist
 an Hamiltonian matrix $A_\om(\theta)$ and a vector $V_\om(\theta)\in\R^{2d}$
   such  that 
   \begin{align}
\lim_{l\rightarrow+\infty}e^{A_l(\theta)} = e^{A_\infty(\theta)} \ , \ \ \ 
\lim_{l\rightarrow+\infty} V_l(\theta) = V_\infty(\theta) \nonumber\\
\sup_{|{\rm Im } \theta| \leq \sigma/2 } \Vert A_\om(\theta) \Vert \leq C\epsilon \ , \ \ \ 
\sup_{|{\rm Im } \theta| \leq \sigma/2} \vert V_\om(\theta)\vert \leq C \epsilon 
\end{align}
and for each $\theta\in\T^n$,
$$\Psi(\theta,x,\xi)=e^{A_\om(\theta)}(x,\xi) + V_\om(\theta) \ \ \  \forall (x, \xi) \in\R^{2d} \ .$$
\end{lemma}
\begin{proof}
Recall  that  $\phi_j =e^{B_j} +T_j$ where  $T_j$ is a translation by the vector $T_j$  with the  estimates
$\Vert B_j\Vert\leq C\epsilon_j$, $ \Vert T_j\Vert \leq C\epsilon_j$.
So  we have $e^{B_j} = \1 +S_j$ with $\Vert S_j\Vert\leq C\epsilon_j$. Then the  infinite product 
$\prod_{1\leq j<+\infty}e^{B_j} $ is convergent. Moreover we have 
$\prod_{1\leq j\leq l}e^{B_j}  = \1 +M_l$  with $\Vert M_l\Vert \leq C\epsilon$  so we have 
$\prod_{1\leq j<+\infty}e^{B_j}  = \1 +M$  with $\Vert M\Vert \leq C\epsilon$. This is proved by using
$$
\prod_{1\leq j\leq l}(\1+S_j) = \1 +S_l +S_{l-1}S_l+\cdots S_1S_2\cdots S_l
$$ and estimates on $\Vert S_j\Vert$. \\ {So, $M_l$ has a small
  norm and therefore $A_l:=\log(\1+M_l)$ is well defined. Furthermore,
  by construction $\1+M_l\in{\rm Sp}(2d)$ and therefore its logarithm
  is a Hamiltonian matrix, namely $A_l\in{\rm sp}(2d)$ for $1\leq
  l\leq +\infty$.}  
\\ Now we have to include the translations.  By induction on $l$ we
have
\begin{equation*}
\phi_{1}\circ...\circ \phi_{l}(x,\xi) = {\rm e}^{A_l}(x,\xi) + V_l \ ,
\end{equation*}
 with $V_{l+1} ={\rm e}^{A_l}T_{l+1} + V_l$  and $V_1 =T_1$. 
Using the previous estimates we have 
$$
\Vert V_{l+1}-V_l\Vert \leq C\Vert T_{l+1}\Vert \leq C\epsilon_{l}. 
$$
Then   we get that $\displaystyle{\lim_{l\rightarrow +\infty}V_l =V_\infty}$ exists.
\end{proof}

\appendix

\section{An example of growth of Sobolev norms (following Graffi and Yajima)}
\label{example}
In this appendix we are going to study the Hamiltonian
\begin{equation}
  \label{graffi}
H:= -\frac{1}{2}\partial_{xx}+\frac{x^2}{2}+ax\sin\omega t
\end{equation}
and prove that it is reducible to the Harmonic oscillator if
$\omega\not = \pm 1$, while the system exhibits growth of Sobolev
norms in the case $\omega=\pm 1$. Actually the result holds in a quite
more general situation, but we think that the present example  can give a full understanding of the situation
with as little techniques as possible. We also remark that in this
case it is not necessary to assume that the time dependent part is
small. 

Finally we recall that \eqref{graffi} with $\omega = \pm 1$ was
  studied by Graffi and Yajima as an example of Hamiltonian whose
  Floquet spectrum is absolutely continuous (despite the fact that the
  unperturbed Hamiltonian has discrete spectrum).  Exploiting the
  results of \cite{enve, buni} one can conclude from \cite{graffi} that
  the expectation value of the energy is not bounded in this
  model. The novelty of the present result rests in the much more
  precise statement ensuring growth of Sobolev norms.

As we already pointed out, 
in order to get reducibility of the Hamiltonian \eqref{graffi}, it is
enough to study the corresponding classical Hamiltonian, in particular
proving its reducibility; this is what we will do. It also turns out that all the procedure is clearer
working as much as possible at the level of the equations.

So, consider the classical Hamiltonian system
\begin{equation}
  \label{clagraffi}
h:=\frac{x^2+\xi^2}{2}+ax\sin(\omega t)\ ,
\end{equation}
whose equations of motion are
\begin{equation}
  \label{eq.gra}
  \left\{
  \begin{matrix}
   & \dot x=\xi
    \\
    &\dot\xi=-x-a\sin(\omega t)
  \end{matrix}
  \right. \iff\quad \ddot x+x+a\sin(\omega t) = 0 \ .
\end{equation}
\begin{proposition}
  \label{prop.gra}
Assume that $\omega\not=\pm 1$.  Then there exists a time periodic
canonical transformation conjugating \eqref{clagraffi} to 
\begin{equation}
  \label{clagraffi2}
h':=\frac{x^2+\xi^2}{2}\ .
\end{equation}
If $\omega=\pm1$ then the system is canonically conjugated to
\begin{equation}
  \label{clagraffi3}
h':=\pm\frac{a}{2}\xi\ .
\end{equation}
In both cases the transformation has the form \eqref{kamclach}\ .
\end{proposition}

\begin{corollary}
  \label{cor.growth}
  In the case $\omega=\pm 1$, for any $s>0$ and $\psi_0 \in \cH^s$, 
  there exists a constant
  $0<C_s = C_s(\norm{\psi_0}_{\cH^s})$ s.t. the solution  of the Schr\"odinger equation with
  Hamiltonian \eqref{graffi} and initial datum $\psi_0$ fulfills
  \begin{equation}
    \label{cre}
\norma{\psi(t)}_{\cH^s}\geq C_s  \langle t\rangle^s ,\ \ \ \forall t\in\R.
  \end{equation}
\end{corollary}

Before proving the theorem, recall that by the general result of \cite[Theorem 1.5]{maro}, any solution of the Schr\"odinger equation with Hamiltonian \eqref{graffi} fulfills the a priori bound 
 \begin{equation}
    \label{cre2}
\norma{\psi(t)}_{\cH^s}\leq C_s' \left( 
\norma{\psi_0}_{\cH^s} + |t|^s \norm{\psi_0}_{\cH^0} \right),\; \forall t\in\R, 
\end{equation}
  which is therefore sharp. 

\begin{proof}[Proof of Proposition \ref{prop.gra}.]  We look for a translation
\begin{equation}
  \label{transl}
x=x'-f(t)\ , \quad \xi=\xi'-g(t)\ ,
\end{equation}
with $f$ and $g$ time periodic functions to be determined in such a way to
eliminate time from \eqref{eq.gra}. Writing the equations for
$(x', \xi')$, one gets
\begin{align*}
\dot x'=\xi'-g+\dot f\ ,\quad \dot \xi'=-x'-a\sin (\omega t)+ \dot g + f \ ,
\end{align*}
which reduces to the harmonic oscillator by choosing
\begin{equation}
  \label{choose}
  \left\{
  \begin{matrix}
    -a\sin(\omega t)+\dot g+f=0
    \\
    -g+\dot f=0
  \end{matrix}
  \right.\iff \ddot f+f=a \sin(\omega t)
\end{equation}
which has a solution of period $2\pi/\omega$ only if
$\omega\not=\pm 1$. In such a case the only solution having the correct period is
$$
f=\frac{a}{1-\omega^2}\sin(\omega t)\ ,\quad g=\frac{a\omega
}{1-\omega^2}\cos(\omega t)  \ .
$$
Then the transformation \eqref{transl} is a canonical transformation
generated as the time one flow of the auxiliary Hamiltonian
$$\chi:=-\xi \frac{a}{1-\omega^2}\sin(\omega t)+ x \frac{a\omega
}{1-\omega^2}\cos(\omega t)
$$
which thus conjugates the classical Hamiltonian \eqref{clagraffi} to the
Harmonic oscillator; of course the quantization of $\chi$ conjugates
the quantum system to the quantum Harmonic oscillator, as follows by Proposition \ref{prop:flow}.\\

We come to the resonant case, and, in order to fix ideas, we take
$\omega=1$. In such a case the flow of the Harmonic oscillator is
periodic of the same period of the forcing, and thus its flow can be
used to reduce the system. 

In a slight more abstract way, consider a Hamiltonian system with
Hamiltonian
$$
H:=\frac{1}{2}\langle z;Bz\rangle+\langle z; b(t)\rangle
$$
with $z:=(x,\xi)$, $B$ a symmetric matrix, and $b(t)$ a vector valued
time periodic function. Then, using the formula \eqref{timetras}, it
is easy to see that the auxiliary time dependent Hamiltonian
\begin{equation}
\label{chi1graffi}
\chi_1:=\frac{t}{2}\langle z;Bz\rangle
\end{equation}
 generates a time periodic transformation which conjugates the
system to
\begin{equation*}
h':=\langle z;e^{-JBt}b(t)\rangle\ 
\end{equation*}
($J$ being the standard symplectic matrix).
An explicit computation shows that in our case
\begin{equation}
  \label{reso.final}
h'=\frac{a}{2} x\sin(2t)-\frac{a}{2}\xi\cos(2t)+\frac{a}{2}\xi\ .
\end{equation}
Then in order to  eliminate the two time
periodic terms in \eqref{reso.final} it is sufficient to use the canonical transformation generated by the Hamiltonian
\begin{equation}
\label{chi2graffi}
\chi_2:=-\xi \frac{a}{4}\sin(2 t)- x\frac{a
}{4}\cos(2t)  \ ,
\end{equation}
which  reduce to \eqref{clagraffi3}.
\end{proof}

\begin{proof}[Proof of Corollary \ref{cor.growth}] To fix ideas we
  take $\omega=1$. 
Let  $\chi_1^w \equiv \frac{t}{2}(-\partial_{xx} + x^2)$ and $\chi_2^w$ be  the Weyl quantization of the Hamiltonians \eqref{chi1graffi} respectively \eqref{chi2graffi}. By the proof of Proposition \ref{prop.gra}, the changes of coordinates 
\begin{equation}
\label{changegraffi}
\psi = e^{-\im  t H_0} \psi_1 \ , \quad \psi_1 = e^{-\im  \chi_2^w(t, x, D)} \vf \ , \qquad H_0 := \frac{1}{2}(-\partial_{xx} + x^2)
\end{equation}
conjugate the Schr\"odinger equation with Hamiltonian \eqref{graffi} to the Schr\"odinger equation with Hamiltonian \eqref{clagraffi}, namely the transport equation
\begin{equation*}
\partial_t \vf = - \frac{a}{2} \partial_x \vf \ .
\end{equation*}
The solution of this transport equation is given clearly by
\begin{equation*}
\vf(t,x) = \vf_0 (x-\frac{a}{2} t)
\end{equation*}
where $\vf_0$ is the initial datum. Now a simple computation shows that 
$$
\liminf_{\vert t\vert\rightarrow +\infty}\vert t\vert^{-s}\norm{\vf(t)}_{\cH^s} \geq \left(\frac{\vert a\vert}{2}\right)^s\norm{\vf_0}_{\cH^0}.
$$
In particular there exists a constant
  $0<C_s = C_s(\norm{\vf_0}_{\cH^s})$ s.t. 
\begin{equation}
\label{est.growth}
\norm{\vf(t)}_{\cH^s} \geq C_s \langle t \rangle^s \ .
\end{equation}
Since the transformation \eqref{changegraffi} maps $\cH^s$ to $\cH^s$ uniformly in time (see also Lemma \ref{lem:flow}) estimate \eqref{est.growth} holds also for the original variables.
\end{proof}

We remark that by a similar procedure one can also prove the following
slightly more general result.

\begin{theorem}
\label{graffigen}
Consider the classical Hamiltonian system
\begin{equation}
\label{gen.gra.1}
h=\sum_{j=1}^{d}\nu_j\frac{x_j^2+\xi_j^2}{2}+\sum_{j=1}^{d}\left(g_j(\omega
  t)x_j+f_j(\omega
  t)\xi_j\right)\ ,
\end{equation}
with $f_j,g_j\in C^r(\T^n)$. 
\begin{itemize}
\item[(1)] If there exist $\gamma>0$ and $\tau>n+1$ s.t. 
\begin{equation}
\label{dio.gra}
\left|\omega\cdot k\pm\nu_j\right|\geq
\frac{\gamma}{1+|k|^\tau}\ ,\quad \forall k\in \Z^n\ ,\quad j=1,...,d 
\end{equation}
and $r>\tau +1+n/2$, then there exists a time quasiperiodic
canonical transformation of the form \eqref{kamclach} conjugating the
system to\footnote{Actually the transformation is just a translation,
so in this case one has $A\equiv 0$.} 
$$
h=\sum_{j=1}^{d}\nu_j\frac{x_j^2+\xi_j^2}{2}\ .
$$
\item[(2)] If there exist $0\not=\bar k\in \Z^n$ and $\bar j$, s.t.
\begin{equation}
\label{resonant}
\omega\cdot \bar k-\nu_{\bar j}=0\ ,
\end{equation}
and there exist $\gamma>0$ and $\tau$ s.t. 
\begin{equation}
\label{dio.gra2}
\left|\omega\cdot k\pm\nu_j\right|\geq
\frac{\gamma}{1+|k|^\tau}\ ,\quad \forall (k,j)\not=(\bar k,\bar j) 
\end{equation}
and $r>\tau+1+\frac{n}{2}$, then there exists a time quasiperiodic
canonical transformation of the form \eqref{kamclach} conjugating the
system to 
$$
h=\sum_{j\not=\bar
  j}\nu_j\frac{x_j^2+\xi_j^2}{2}+c_1x_{\bar j}+c_2\xi_{\bar j}\ ,
$$
with $c_1,c_2\in\R$.
\end{itemize}
\end{theorem}
\begin{remark}
\label{rem.gra.gen}
The constants $c_1,c_2$ can be easily computed. If at least one
of them is different form zero then the solution of the
corresponding quantum system exhibits growth of Sobolev norms as in
the special model \eqref{graffi}. Of course the result extends in a
trivial way to the case in which more resonances are present.
\end{remark}

\newcommand{\etalchar}[1]{$^{#1}$}
\def\cprime{$'$}

\end{document}